\documentclass{article}
\usepackage{amsmath,amssymb,latexsym,enumerate}
\usepackage{amsfonts}
\usepackage{theorem}
\usepackage{array}
\usepackage[a4paper]{geometry}
\usepackage{graphics}
\usepackage[usenames]{color}
\usepackage[all]{xypic}
\usepackage{graphicx}
\usepackage{boxedminipage}

\newtheorem{theorem}{Theorem}

\newtheorem{corollary}[theorem]{Corollary}

\newtheorem{lemma}[theorem]{Lemma}
\newtheorem{remark}[theorem]{Remark}

\newtheorem{proposition}[theorem]{Proposition}

\newenvironment{proof}[1][Proof]{\textbf{#1.} }{\ \rule{0.5em}{0.5em}}

\newcommand{\F}{\mathbb{F}}

\newcommand{\Tr}{{\rm{Tr}}}

\begin{document}

\title{Galois Towers over Non-prime Finite Fields}
\author{Alp Bassa\footnote{Alp Bassa is supported by T\"{u}bitak Proj. No. 112T233.}, Peter Beelen\footnote{Peter Beelen gratefully acknowledges the support from the Danish National Research Foundation and the National Science Foundation of China (Grant No.11061130539) for the Danish-Chinese Center for Applications of Algebraic Geometry in Coding Theory and Cryptography.
}, Arnaldo Garcia\footnote{Arnaldo Garcia is partially supported by CNPq (Brazil).} and Henning Stichtenoth\footnote{Henning Stichtenoth is supported by T\"{u}bitak Proj. No. 111T234.}}
\date{\empty}

\maketitle
\begin{abstract}
In this paper we construct Galois towers with good asymptotic properties over any non-prime finite field $\mathbb F_{\ell}$; i.e.,  we construct sequences of function fields $\mathcal{N}=(N_1 \subset N_2 \subset \cdots)$ over $\mathbb F_{\ell}$ of increasing genus, such that all the extensions $N_i/N_1$ are Galois extensions and the number of rational places of these function fields grows linearly with the genus. The limits of the towers satisfy the same lower bounds as the best currently known lower bounds for the Ihara constant for non-prime finite fields. Towers with these properties are important for applications in various fields including coding theory and cryptography.
\end{abstract}

\section{Introduction}

The question of how many rational points a curve over a finite field can have is not only interesting from a purely number-theoretical perspective, but also has become an important question for applications in computer science, coding theory, cryptography and other areas of discrete mathematics.
Curves with many rational points have been successfully applied in the construction of codes, sequences, hash functions, secret sharing and multiparty computation schemes and other combinatorial objects. One of the landmark results in this direction is the work of Tsfasman--Vladut--Zink \cite{TVZ}, where sequences of curves of increasing genus with good asymptotic behavior and a construction of codes from curves with many points due to Goppa are combined to construct codes better than the Gilbert--Varshamov bound. This was a big surprise, as the Gilbert--Varshamov bound had resisted any attempt of improvement for many years.

Although several such sequences of curves with the same good asymptotic behavior exist, some turn out to be more suitable for applications than others. Recent work has shown that various additional properties enjoyed by the curves in some of these sequences turn out to be very beneficial for applications. These additional properties satisfied by the curves in the sequence reflect themselves in further features or better parameters of the objects constructed from them. For instance, Stichtenoth \cite{trans} showed how sequences of curves with many points together with the additional property that each of them is a Galois covering of the first one can be used to construct self-dual and transitive codes attaining the Tsfasman--Vladut--Zink bound. Also, in \cite{CCX} Cascudo, Cramer and Xing showed how, in the construction of arithmetic secret sharing schemes from sequences of curves with many rational points, a better control on the $d$-torsion in the class group of the curves involved leads to better bounds for the constructed schemes (see also \cite{hassewitt}).

With these and similar applications in mind, we construct in this paper over any non-prime finite field $\mathbb F_{\ell}$ sequences of curves with increasing genus and many rational points, such that each curve in the sequence is a Galois covering of the first one. Instead of the geometric language of curves over finite fields, we will use the equivalent language of algebraic function fields with finite constant fields. So, more precisely, over any non-prime finite field $\mathbb F_{\ell}$ we will construct sequences of function fields $\mathcal{N}=(N_1 \subset N_2 \subset \cdots)$ such that for each $i>0$ the extension $N_i/N_1$ is a Galois extension and moreover  $\mathcal{N}$ has a large limit.
For a more precise statement, see Theorem~\ref{thm:collecting} below.

Let $\mathcal G =(G_1 \subset G_2 \subset \cdots)$ be a sequence of functions fields with full constant field $\mathbb F_{\ell}$. Such a sequence is called a tower over $\mathbb F_{\ell}$.  Let $f(x,y) \in \mathbb F_{\ell}[x,y]$. We say that the tower $\mathcal G$ satisfies the equation $f(x,y)=0$ recursively, if for all $i\geq 1$ there exists $x_i\in G_i$ such that
\begin{itemize}
\item $x_1$ is transcendental over $\mathbb F_{\ell}$,
\item $G_i=G_{i-1}(x_i)$ and $f(x_{i-1},x_i)=0$ for $i>1$.
\end{itemize}
Such a tower is simply called a recursive tower. The main ingredients for this paper are the recursive towers that were introduced by the authors in \cite{BBGS}.

For a function field $F$ over $\mathbb{F}_\ell$ we denote by $N(F)$ the number of rational paces and by $g(F)$ its genus.
Let $q$ be a power of a prime $p$, $1 \le k <n$ be integers such that $\gcd(k,n-k)=1$, and let $\ell=q^n$. In \cite{BBGS} we introduced and studied the towers $\mathcal F =(F_1 \subset F_2 \subset \cdots)$ over $\mathbb{F}_{\ell}$ satisfying the recursive equation
\begin{equation}
\label{tow1i}
\frac{y}{x^{q^k}}+\frac{y^q}{x^{q^{k+1}}}+\cdots+\frac{y^{q^{n-k-1}}}{x^{q^{n-1}}}+\frac{y^{q^{n-k}}}{x}+\frac{y^{q^{n-k+1}}}{x^q}+\cdots+\frac{y^{q^{n-1}}}{x^{q^{k-1}}}=1.
\end{equation}
We showed that the limit $$\lambda(\mathcal F):=\lim_{i\to \infty}\frac{N(F_i)}{g(F_i)}$$ of this tower satisfies
\begin{equation}
\label{limitbound}
\lambda(\mathcal F) \ge 2\left(\dfrac{1}{q^{k}-1}+\dfrac{1}{q^{n-k}-1}\right)^{-1}.
\end{equation}

Consider a tower $\mathcal G =(G_1 \subset G_2 \subset \cdots)$ over $\mathbb F_{\ell}$. Assume that for all $i\ge 1$ the extensions $G_{i+1}/G_i$ are separable (hence so are the extensions $G_i/G_1$). Let  $\tilde{G_i}$ be the Galois closure of the extension $G_i/G_1$ and assume that $\mathbb F_{\ell}$ is algebraically closed in all $\tilde{G_i}$. The tower  $\tilde{\mathcal G} =(\tilde{G_1} \subset \tilde{G_2} \subset \cdots)$ is called the Galois closure of $\mathcal G$.

In this paper we investigate the Galois closure of the tower $\mathcal F$ and of some of its subtowers introduced in \cite{BBGS}.
We investigate the splitting and ramification behavior of places in these towers, study the Galois groups of the extensions and show that each of these Galois towers has a limit satisfying Inequality~\eqref{limitbound}. Along the way, we also show that there exists a finite extension $E/F_1$ such that each step in the composite tower $E\mathcal F=(EF_1 \subset EF_2 \subset \cdots)$ is Galois with an elementary abelian $p$-group as Galois group. We collect the main results of this paper in the following theorem:

\begin{theorem}\label{thm:collecting}
Let $q$ be a prime power. For any integer $n>1$ and $1 \le k <n$ with $\gcd(k,n-k)=1$ there exists a tower $\mathcal{N}=(N_1 \subset N_2 \subset \cdots)$ over $\F_\ell$, where $\ell=q^n$, such that
\begin{enumerate}[i)]
\item $N_1=\F_\ell(z_1)$ is a rational function field.
\item For each $i\geq 2$, the extension $N_i/N_1$ is a Galois extension having as Galois group an extension of a subgroup of $\mathrm{GL}_{n-1}(\F_q)$ by a $p$-group. The extension $N_i/N_2$ is a $p$-extension.
\item The place $[z_1=-1]$ of $N_1$ splits completely in $\mathcal{N}$; i.e., it splits completely in each extension $N_i/N_1$.
\item The only places of $N_1$ which are ramified in $\mathcal{N}$ are $P_0:=[z_1=0]$ and $P_\infty:=[z_1=\infty]$, and they are weakly ramified (i.e., their second ramification groups are trivial).
\item For each $i>1$, the extension $N_i/N_2$ is $2$-bounded; more precisely, for any place $P$ of $N_2$  the ramification index $e(P)$ and different exponent $d(P)$ of $P$ in the extension $N_i/N_2$ satisfy
$$d(P)=2(e(P)-1).$$
\item Let $e_i(P_0)$ and $e_i(P_\infty)$ denote the ramification indices in the extension $N_i/N_1$ of the places $P_0$ and $P_\infty$ respectively and assume that $i>1$. We have
    $$e_i(P_0)=(q^k-1)q^{(i-1)(n-k)-k}p^{\epsilon_1(i)}$$
and
$$e_i(P_\infty)=(q^{n-k}-1)q^{(i-1)(n-k)}p^{\epsilon_2(i)},$$
with $\epsilon_1(i),\epsilon_2(i) \ge 0$.
\item The limit of the tower satisfies $$\lambda(\mathcal{N}) \ge 2\left(\frac{1}{q^k-1}+\frac{1}{q^{n-k}-1}\right)^{-1}.$$
\end{enumerate}
\end{theorem}

\section{Preliminaries}\label{sec:one}

In this section we establish some preliminaries and recall some notations and results from \cite{BBGS}.

Throughout the rest of the paper, $q$ will be a power of a prime $p$ and $\ell=q^n$ for some $n \geq 2$. Let $E/F$ be a Galois extension of function fields over $\mathbb{F}_\ell$. Let $P$ be a place of $F$ and $Q$ a place of $E$ lying over $P$. We say that $Q|P$ is weakly ramified, if $G_2(Q|P)=\{e\}$, where $G_2(Q|P)$ denotes the second ramification group of $Q|P$. The Galois extension $E/F$ is said to be weakly ramified, if for all places $P$ of $F$ and all places $Q$ lying above $P$, $Q|P$ is weakly ramified.
A weakly ramified $p$-extension $E/F$ is $2$-bounded. For such an extension, for every place $P$ of $F$ and every place $Q$ above $P$ we have $d(Q|P)=2\cdot (e(Q|P)-1)$.

For convenience we define for any positive integer $i$ the \emph{trace polynomial}
$$\Tr_i(x)= x+x^q+x^{q^2}+\cdots+x^{q^{i-1}}$$
The trace polynomials $\Tr_i(x)$ are examples of $q$-additive polynomials. The following lemma will be useful later on:
\begin{lemma}
\label{lem1}
Let $i$ and $j$ be positive integers.
\begin{enumerate}[i)]
\item We have
$$\Tr_i(\Tr_j(x))=\Tr_j(\Tr_i(x)).$$
More generally, any two $q$-additive polynomials with coefficients in $\mathbb F_q$ commute.
\item Setting $r=\gcd(i,j)$, for any field $L\supset \mathbb F_q$we have
$$L(\Tr_i(x),\Tr_j(x))=L(\Tr_r(x))\subseteq L(x).$$
In particular, if $\gcd(i,j)=1$, then $L(\Tr_i(x),\Tr_j(x))=L(x)$.
\end{enumerate}
\end{lemma}
\begin{proof}
The first part follows by a direct computation. For the second part we assume w.l.o.g. that $i>j$ (the case $i=j$ is trivial).
Then
$$\Tr_i(x)=\Tr_{i-j}(x)+(\Tr_j(x))^{q^{i-j}},$$
so
$$L(\Tr_i(x),\Tr_j(x))=L(\Tr_j(x),\Tr_{i-j}(x)).$$
The claim then follows from the properties of the Euclidean Algorithm.
\end{proof}

The second claim of the lemma is equivalent to saying that $\Tr_r(x)$ can be expressed in terms of $\Tr_i(x)$ and $\Tr_j(x)$. This can be shown more explicitly: Let $a$ and $b$ be positive integers such that $a i-b j=r$ (note that such $a$ and $b$ always exist). Then $\Tr_{ai}(x)-\Tr_{bj}(x)^{q^r}=\Tr_r(x),$ which implies that
\begin{equation}
\label{expleuclid}
\sum_{\alpha=0}^{a-1}\Tr_i(x)^{q^{\alpha i}}-\left(\sum_{\beta=0}^{b-1}\Tr_j(x)^{q^{\beta j}}\right)^{q^r}=\Tr_r(x).
\end{equation}
Now let $0<k<n$ with $\gcd(n,k)=1$ be given. Let $a,b$ be non-negative integers such that
\begin{equation}
\label{eq1}
a\cdot k-b\cdot (n-k)=1.
\end {equation}

Suppose $x$ and $y$ satisfy Equation \eqref{tow1i} and let $$R:=\frac{y}{x^{q^k}} \qquad \textrm{and} \qquad S:=\frac{y^{q^{n-k}}}{x}.$$ The quantities $R$ and $S$ occur in Equation \eqref{tow1i}:
$$
\underbrace{\frac{y}{x^{q^k}}}_{R}+\frac{y^q}{x^{q^{k+1}}}+\cdots+\underbrace{\frac{y^{q^{n-k-1}}}{x^{q^{n-1}}}}_{R^{q^{n-k-1}}}+\underbrace{\frac{y^{q^{n-k}}}{x}}_S+\frac{y^{q^{n-k+1}}}{x^q}+\cdots+\underbrace{\frac{y^{q^{n-1}}}{x^{q^{k-1}}}}_{S^{q^{k-1}}}=1.
$$
And therefore we obtain
\begin{equation}\label{eq:RS}
\Tr_{n-k}(R)+\Tr_k(S)=1.
\end{equation}
\begin{proposition}
The function field $\F_{\ell}(R,S)$ is a rational function field.
More precisely, letting
$$u:=\sum_{\alpha=0}^{a-1}R^{q^{\alpha k}}+\left(\sum_{\beta=0}^{b-1}S^{q^{\beta(n-k)}}\right)^{q},$$
we have $R=\Tr_k(u)-b$, $S=-\Tr_{n-k}(u)+a$ and hence $\F_{\ell}(R,S)=\F_{\ell}(u)$.
\end{proposition}
\begin{proof}
We have
\begin{align*}
\Tr_k(u) & =  \sum_{\alpha=0}^{a-1}\Tr_k(R)^{q^{\alpha k}}+\left(\sum_{\beta=0}^{b-1}\Tr_k(S)^{q^{\beta(n-k)}}\right)^{q} & \\
& =  \sum_{\alpha=0}^{a-1}\Tr_k(R)^{q^{\alpha k}}+\left(\sum_{\beta=0}^{b-1}(1-\Tr_{n-k}(R))^{q^{\beta(n-k)}}\right)^{q} & \makebox{by Equation \eqref{eq:RS}}\\
& = \Tr_{ak}(R)-\Tr_{b(n-k)}(R)^q+b\\
& = R+b. & \makebox{by Equation \eqref{eq1}}
\end{align*}
Similarly $\Tr_{n-k}(u)=-S+a$. It follows that $\F_{\ell}(R,S)=\F_{\ell}(u)$.
\end{proof}

From the above it is clear how to express $u$ explicitly in terms of $x$ and $y$. Note that
\begin{equation}
\label{eq:relations}
y^{q^n-1}=\frac{S^{q^k}}{R}=-\frac{\Tr_{n-k}(u)^{q^k}-a}{\Tr_k(u)-b} \ {\rm{and}} \ x^{q^n-1}=\frac{S}{R^{q^{n-k}}}=-\frac{\Tr_{n-k}(u)-a}{\Tr_k(u)^{q^{n-k}}-b}.
\end{equation}
It was shown in \cite[Lemma 2.9]{BBGS} that $\F_{\ell}(x^{q^n-1},y^{q^n-1})=\F_{\ell}(u)$. Therefore, one can express $u$ not only as a rational expression in $x$ and $y$, but also in $x^{q^n-1}$ and $y^{q^n-1}$, say $$u=\phi(x^{q^n-1},y^{q^n-1}).$$

Now let $\mathcal F =(F_i)_{i>0}$ be a tower over $\F_{\ell}$, where $F_1=\F_{\ell}(x_1)$ is a rational function field and for all $i>1$, there exist $x_i \in F_i$ such that  $F_{i}=F_{i-1}(x_{i})$ with
\begin{equation}
\label{tow1}
\frac{x_{i}}{x_{i-1}^{q^k}}+\frac{x_{i}^q}{x_{i-1}^{q^{k+1}}}+\cdots+\frac{x_{i}^{q^{n-k-1}}}{x_{i-1}^{q^{n-1}}}+\frac{x_{i}^{q^{n-k}}}{x_{i-1}}+\cdots+\frac{x_{i}^{q^{n-1}}}{x_{i-1}^{q^{k-1}}}=1.
\end{equation}
Thus $\mathcal F$ satisfies the recursion given by Equation~\eqref{tow1i}.

Defining $u_i:=\phi(x_i^{q^n-1},x_{i+1}^{q^n-1})$ and $z_i:=-x_i^{q^n-1}$, we see from Equation \eqref{eq:relations} that
\begin{equation}
\label{tow2}
z_{i}=-x_i^{q^n-1}=\frac{\Tr_{n-k}(u_{i})-a}{\Tr_k(u_{i})^{q^{n-k}}-b}=\frac{\Tr_{n-k}(u_{i-1})^{q^k}-a}{\Tr_k(u_{i-1})-b}.
\end{equation}

Consider the subtowers $\mathcal E =(E_i)_{i> 0}$ and $\mathcal H=(H_i)_{i> 0}$ of $\mathcal F$ where $E_i=\F_{\ell}(u_1,\dots,u_i)=\F_{\ell}(z_1,\dots,z_{i+1})$ and $H_i=\F_{\ell}(z_1,\dots,z_i)$. Note that for $i>0$ we have $E_i=H_{i+1}$. See Figure \ref{fig:one} for a graphical overview of the fields occurring in $\mathcal F$, $\mathcal E$ and $\mathcal H$. From Equation \eqref{tow2} we see that the tower $\mathcal E$ satisfies a recursive equation. In \cite[Equation (38)]{BBGS} we gave a recursive equation satisfied by the tower $\mathcal H$.

\begin{remark}\label{rem:composite}
It was shown in \cite{BBGS} that $E_i(x_1)=F_{i+1}$. This means that the tower $\mathcal F$ can be seen as the composite of the tower $\mathcal H$ and the field $F_1$.
\end{remark}

\begin{figure}[h]
\begin{center}
\scalebox{1.0}{
\makebox[\width][c]{
\xymatrix@!=2.5pc@dr{
\framebox{{\Large $\mathcal F$}}&\framebox{{\Large $\mathcal H, \mathcal E$}}&&&&\\
F_{4}=E_3(x_1) \ar@{-}[r] \ar@{.}[u]& H_4=E_3=\F_{\ell}(u_1,u_2, u_3) \ar@{-}[r] \ar@{.}[u]&K(u_2,u_3)\ar@{-}[r]\ar@{.}[u]& K(u_3) \ar@{.}[u] \ar@{-}[r]& \\
F_{3}=E_2(x_1) \ar@{-}[r] \ar@{-}[u]& H_3=E_2=\F_{\ell}(u_1,u_2) \ar@{-}[u] \ar@{-}[r]&  \F_{\ell}(u_2) \ar@{-}[u] \ar@{-}[r] & \F_{\ell}(z_3) \ar@{-}[u]\\
F_2=E_1(x_1) \ar@{-}[r] \ar@{-}[u]& H_2=E_1=\F_{\ell}(u_1) \ar@{-}[u] \ar@{-}[r]&\F_{\ell}(z_2) \ar@{-}[u]\\
F_1=\F_{\ell}(x_1) \ar@{-}[u]\ar@{-}[r] & H_1=\F_{\ell}(z_1)\ar@{-}[u]
}
}
}
\end{center}
\caption{The towers $\mathcal F=(F_i)_{i>0}$, $\mathcal E=(E_i)_{i>0}$ and $\mathcal H =(H_i)_{i > 0}$.\label{fig:one}}
\end{figure}
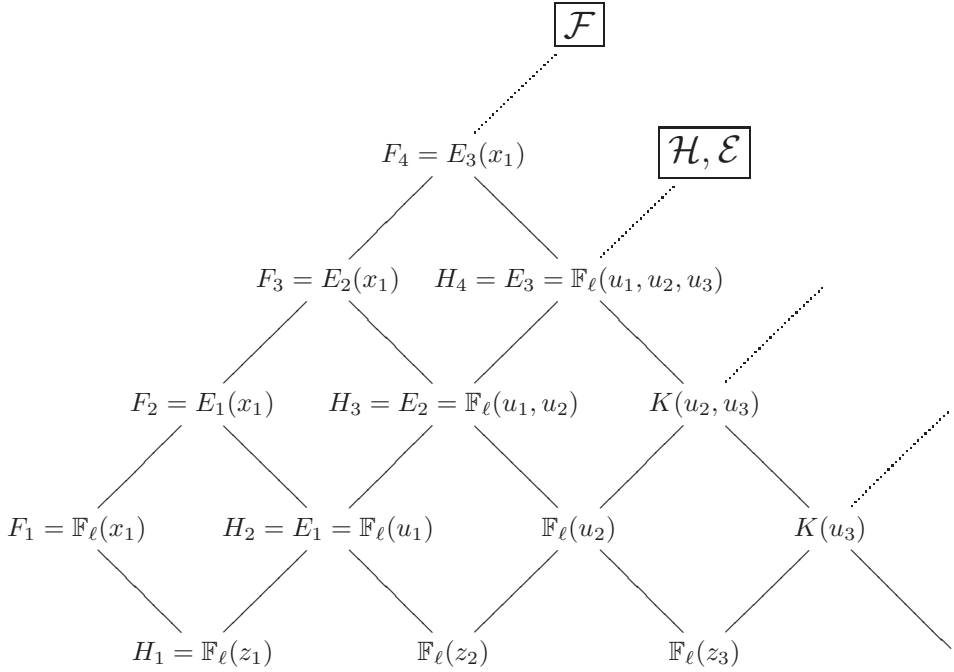

\begin{remark} Let $\mathcal F$ be a tower satisfying a recursion $f(x,y)=0$. Define the dual polynomial $\hat{f}(x,y):=f(y,x)$. A tower $\hat{\mathcal F}$ satisfying the recursion $\hat{f}(x,y)=0$ is called a dual tower of $\mathcal F$.

Let $\hat{\mathcal E}$ be a dual tower of the tower $\mathcal E$ defined above. The towers $\mathcal E$ and $\hat{\mathcal E}$ have very similar behavior. Equation \eqref{tow2} implies that the tower $\hat{\mathcal E}$ satisfies the recursive equation
\begin{equation}
\label{tow2d}
\frac{\Tr_k(u_{r})-b}{\Tr_{n-k}(u_{r})^{q^k}-a}=\frac{\Tr_k(u_{r-1})^{q^{n-k}}-b}{\Tr_{n-k}(u_{r-1})-a}.
\end{equation}
This equation is obtained from Equation \eqref{tow2} by interchanging both $k$ with $n-k$ and $a$ with $b$.
\end{remark}

\begin{remark}
If $\gcd(n-k,p)=1$, we can choose $a \equiv 0 \pmod{p}$ in Equation \eqref{eq1}. The corresponding choice of $b$ will satisfy $b\cdot(n-k) \equiv -1 \pmod{p}.$ Equation~\eqref{tow2} then gets the form
$$\frac{\Tr_{n-k}(u_{i+1})}{\Tr_k(u_{i+1})^{q^{n-k}}+\alpha}=\frac{\Tr_{n-k}(u_{i})^{q^k}}{\Tr_k(u_{i})+\alpha},$$
with $\alpha:=(n-k)^{-1} \in \F_p$. In this form the subtower $\mathcal E \subset \mathcal F$ appeared in \cite{BBGS}.
\end{remark}
Next we collect some facts about the tower $\mathcal H$.
\begin{proposition}\label{prop:splittingH}
The place $[z_1=-1]$ of $H_1$ splits completely in the tower $\mathcal H$.
\end{proposition}
\begin{proof}
This follows from \cite[Corollary 3.2]{BBGS} and the fact that $z_1=-x_1^{q^n-1}$.
\end{proof}

While investigating ramification, we replace the constant field $\F_{\ell}$ by its algebraic closure $K:=\overline{\F}_\ell$. Moreover, for completions, since the place at which we complete is clear from the context, we do not specify the place explicitly in the notation. A place and the corresponding maximal ideal of the valuation ring in the completion are by slight abuse of notation denoted by the same symbol. 

\begin{proposition}\label{prop:ramificationH}
Let $i>0$ and $Q$ be a place of $H_i$, let $P=Q \cap H_1$ be its restriction to $H_1$. If $Q|P$ is ramified, then one of the following holds:
\begin{enumerate}
\item There exists $1 \le m <i $ such that $z_1(Q)=\cdots=z_m(Q)=0$ and $z_{m+1}(Q)=\cdots=z_i(Q)=\infty$. Completing various fields at $Q$ and its restrictions, there is an intermediate field $L$ of the extension $\widehat{H_i}/\widehat{H_1}$ such that $L/\widehat{H_1}$ is cyclic of degree $q^k-1$ and $\widehat{H_i}/L$ can be divided into $2$-bounded elementary abelian $p$-extensions.

\item One has $z_1(Q)=\infty$ and $e(Q|P)=q^{(n-k)(i-1)}$. Let $t_0 \neq 0$ be chosen such that $\Tr_{n-k}(t_0)^{q^k}-z_1\Tr_k(t_0)=0$
and choose a place $P'$ of $K(t_0)$ such that $t_0(P')=\infty$. Suppose that there exists a place $Q'$ of $H_i(t_0)$ lying above both $P'$ and $Q$.
\begin{enumerate}
\item  Completing various fields at $Q'$ and its restrictions, there is an intermediate field $G_1$ of the extension $\widehat{K(t_0)}/\widehat{K(z_{1})}$ such that $G_1/\widehat{K(z_{1})}$ is cyclic of degree $q^{n-k}-1$ and $\widehat{K(t_0)}/G_1$ is a $2$-bounded elementary abelian $p$-extension.
\item Letting $G_j$ be $G_1 \widehat{H_j}$ for $1 \le j \le i$, the extensions $G_{j+1}/G_j$ are $2$-bounded elementary abelian $p$-extensions for $1 \le j < i$.
\end{enumerate}
\end{enumerate}
\end{proposition}
\begin{proof}
The fact that the ramification locus of the tower $\mathcal H$ only consists of the zero and the pole of $z_1$ is a direct consequence of \cite[Proposition 2.6]{BBGS}. The first part about the zero of $z_1$ follows from \cite[Propositions 3.5, 3.6 and Figure 14]{BBGS}. The second part can be shown very similarly to these propositions. The only difference with \cite[Propositions 3.5 and 3.6]{BBGS} is that the element $t_0$ satisfies the equation $\Tr_{n-k}(t_0)^{q^k}-z_1\Tr_k(t_0)=0$, while the element $u$ mentioned there satisfies $\Tr_{n-k}(u)^{q^k}-z_2\Tr_k(u)=a-bz_2.$
\end{proof}

\section{The Galois closure of the tower $\mathcal H$}

Let us denote by $N_i$ the Galois closure of the extension $H_i/H_1$. It follows easily that $\mathbb F_{\ell}$ is algebraically closed in all $N_i$, since there exists a rational place of $H_1$ splitting completely in the extension $H_i/H_1$ (see \cite[Proposition 14]{galclos}). By definition, the tower $\mathcal{N}=\tilde{\mathcal{H}}=(N_1 \subset N_2 \subset \dots)$ is the Galois closure of $\mathcal H$ (over $N_1=H_1$). It is a Galois tower; i.e., each extension $N_i/N_1$ is a Galois extension. We will now study the limit of $\mathcal{N}$ and show that it satisfies Inequality~\eqref{limitbound}.

The field $N_i$ is obtained by taking the composite of several conjugates $\sigma(H_i)$ of $H_i$, with $\sigma$ an element of the absolute Galois group of $H_1$. Since the ramification behavior in the extension $\sigma(H_i)/H_1$ is similar to that of $H_i/H_1$, the analysis of the tower $\mathcal H$ in \cite{BBGS} as described in Proposition \ref{prop:ramificationH} will be very useful. We start by studying the Galois closure of the extension $H_2/H_1$. We define the polynomials
\begin{equation}
\label{hom}
f(T)  :=  -z_1^{-1}\Tr_{n-k}(T)+\Tr_k(T)^{q^{n-k}}
\end{equation}
and
\begin{equation}
\label{homd}
g(T):=\Tr_{n-k}(T)^{q^k}-z_1\Tr_k(T).\\
\end{equation}

\begin{proposition}\label{prop:galoisclosurestep1}
The Galois closure of $H_2/H_1$ is equal to the composite of $H_2$ and the splitting field of $f(T)$ over $H_1$.
\end{proposition}
\begin{proof}
The Galois closure of $H_2/H_1$ is obtained by adjoining to $H_2$ all roots of the polynomial $\Tr_{n-k}(T)-z_1\Tr_k(T)^{q^{n-k}}-a+bz_1$, or equivalently, all roots of the polynomial $f(T)+az_1^{-1}-b$. However, the differences of two roots $u,v$ of $f(T)+az_1^{-1}-b$ are exactly the roots of $f(T)$.
\end{proof}

The polynomial $g(T)$ plays the same role for a dual tower of $\mathcal H$ as the polynomial $f(T)$ does for $\mathcal H$. We will show in Proposition~\ref{thm:splittingfields} that the splitting fields of $f(T)$ and $g(T)$ are the same, which is a fact we will use later. To show this we need the following result (see \cite[Theorem 1.7.11]{Goss}):
\begin{proposition}\label{lem:goss}
Let $F$ be a field containing $\F_q$ and $h(T)=\sum_{i=0}^t a_i T^{q^i} \in F[T]$ be a $q$-additive polynomial with $a_0 \neq 0$ and $a_t \neq 0$. Define $h^{\rm ad}(T):=\sum_{i=0}^t a_i^{q^{t-i}}T^{q^{t-i}}$. Then the roots of $h(T)$ and $h^{\rm ad}(T)$ generate the same extension of $F$.
\end{proposition}

A direct consequence of this proposition is that the extension of $\F_q(z_1)$ generated by the roots of $f(T)$, is the same as the extension of $\F_q(z_1)$ generated by the roots of
$$(z_1f)^{\rm ad}(T)=-( T^{q^{n-1}}+\cdots+T^{q^{k}}) +z_1^{q^{k-1}}T^{q^{k-1}}+\cdots+z_1^qT^q+z_1T.$$
To relate the roots of $f(T)$ with those of $g(T)$, we will use the following lemma:
\begin{lemma}\label{lem:tracek}
Let $t$ be a root of $g(T)$, then $\Tr_k(t)$ is a root of $(z_1f)^{\rm ad}(T)$.
\end{lemma}
\begin{proof}
Since $g(t)=0$, we have $\Tr_{n-k}(t)^{q^{k}}=z_1\Tr_k(t)$. Applying $\Tr_k$ and using Lemma \ref{lem1}, we obtain
$$\Tr_{n-k}(\Tr_k(t))^{q^{k}}=\Tr_k(z_1 \Tr_k(t)).$$ This proves that $\Tr_k(t)$ is a root of $(z_1f)^{\rm ad}(T)$.
\end{proof}

\begin{proposition}\label{thm:splittingfields}
The splitting fields of the polynomials $f(T)$ and $g(T)$ over $H_1$ are the same.
\end{proposition}
\begin{proof}
Using Proposition \ref{lem:goss} we are done once we show that the roots of $g(T)$ and $(z_1f)^{\rm ad}(T)$ generate the same extension. We denote by $V$, respectively $W$, the $\F_q$-vector space consisting of the $q^{n-1}$ roots of $g(T)$, respectively of $(z_1f)^{\rm ad}(T)$. Lemma \ref{lem:tracek} gives rise to an $\F_q$-linear map $\psi$ from $V$ to $W$ defined by $\psi(t)=\Tr_k(t)$. The proposition follows if we show that the map $\psi$ has trivial kernel. Suppose therefore that $\Tr_k(t)=0$. Since $g(t)=0$ as well, one obtains that $\Tr_{n-k}(t)=0$. Using Equations \eqref{expleuclid} and \eqref{eq1}, we see that $t=0$.
\end{proof}

\begin{remark}\label{rem:splittingingalois}
As an immediate consequence of Proposition \ref{prop:galoisclosurestep1} and Proposition \ref{thm:splittingfields} we see that all roots of $f(T)$ and $g(T)$ are contained in $N_i$ for $i \ge 2$.
\end{remark}

These facts will be used to determine the ramification behavior in the tower $\mathcal{N}$. Let $P$ be a place of $H_1$ ramified in $N_i/H_1$. Since the sets of places of $H_1$ that ramify in $N_i/H_1$ and $H_i/H_1$ agree, $P$ is either the pole or the zero of $z_1$ by Proposition~\ref{prop:ramificationH}. Let $\tilde{Q}$ be a place of $N_i$ lying above such a place $P$. We have the following proposition about the ramification of $\tilde{Q}|P$:

\begin{proposition}\label{prop:ramificationbehavior}
Completing $N_i$ at $\tilde{Q}$, there exists an intermediate field $L$ of $\widehat{N_i}/\widehat{N_1}$ such that the extension $L/\widehat{N_1}$ is cyclic and the extension $\widehat{N_i}/L$ is a $2$-bounded $p$-extension. If $P$ is the zero of $z_1$, then $[L:\widehat{N_1}]=q^k-1$. If $P$ is the pole of $z_1$, we have  $[L:\widehat{N_1}]=q^{n-k}-1$.
\end{proposition}
\begin{proof}
Denote by $Q_1,\dots,Q_s$ be the restrictions of $\tilde{Q}$ to the various conjugates $\sigma_1(H_i),\dots,\sigma_s(H_i)$ of $H_i$. We will consider the two cases $z_1(P)=0$ and $z_1(P)=\infty$ separately.

\bigskip
\noindent
\emph{Case i) $z_1(P)=0$:}

\noindent
From the first part of Proposition \ref{prop:ramificationH} we see that after completion at $\tilde{Q}$, the extensions $\widehat{\sigma_j(H_i)}/\widehat{H_1}$ all can be divided into a cyclic part of degree $q^k-1$ and steps of $2$-bounded elementary abelian $p$-extensions. Taking composites we see (using Abhyankar's lemma and \cite[Proposition 12]{galclos}) that there exists a field $L\subset \widehat{N_i}$ such that the extension $L/\widehat{H_1}$ is cyclic of degree $q^k-1$ and such that the extension $\widehat{N_i}/L$ can be divided into $2$-bounded elementary abelian $p$-extensions.

\bigskip
\noindent
\emph{Case ii) $z_1(P)=\infty$:}

\noindent Let $t_0$ be a nonzero root of $g(T)$. By Remark \ref{rem:splittingingalois} the element $t_0$ is contained in $N_2$ and hence in $N_i$. Let $P'$ be a place of $H_1(t_0)$ lying above $P$ such that $t_0(P')=\infty$ and $\tilde{R}$ a place of $N_i$ lying above $P'$. We denote the restrictions of $\tilde{R}$ to the conjugates $\sigma_1(H_i),\dots,\sigma_s(H_i)$ of $H_i$ by $R_1,\dots,R_s$ and the restrictions to $\sigma_1(H_i(t_0)),\dots,\sigma_s(H_i(t_0))$ by $R_1',\dots,R_s'$. The second part of Proposition \ref{prop:ramificationH} implies that after completion at $\tilde{R}$, the extensions $\widehat{\sigma_j(H_i(t_0))}/\widehat{H_1}$ all can be divided into a cyclic part of degree $q^{n-k}-1$ and steps of $2$-bounded elementary abelian $p$-extensions. Again, using Abhyankar's lemma and \cite[Proposition 12]{galclos}, we obtain the desired result for the place $\tilde{R}$. Since $N_i/H_1$ is a Galois extension and $\tilde{Q}$ and $\tilde{R}$ lie above the same place $P$ of $H_1$, the same holds for $\tilde{Q}$.
\end{proof}

\begin{proposition}
Let $e_i(P_0)$ and $e_i(P_\infty)$ denote the ramification indices in the extension $N_i/N_1$ of the places $P_0$ and $P_\infty$ respectively. Then for $i>1$ we have
$$e_i(P_0)=(q^k-1)q^{(i-1)(n-k)-k}p^{\epsilon_1(i)}$$
and
$$e_i(P_\infty)=(q^{n-k}-1)q^{(i-1)(n-k)}p^{\epsilon_2(i)}$$
with $\epsilon_1(i),\epsilon_2(i) \ge 0$.
\end{proposition}
\begin{proof}
We first consider the case of the place $P_0$. We will give a lower bound for the ramification by estimating the highest ramification index among all places of $H_i$ lying over $P_0$. Since $N_i/H_1$ is a Galois extension, the ramification index $e(\tilde{Q}|P_0)$ does not depend on the choice of the place $\tilde{Q}$ of $N_i$ lying over $P_0$. Without loss of generality we may therefore assume that $z_2(\tilde{Q})=\infty$.

Let $Q$ be the restriction of $\tilde{Q}$ to $H_i$ and extend the constant field to $K:=\overline{\mathbb{F}_\ell}$. We will use the notation from \cite{BBGS}, especially the notation occurring in Figures 9 and 11 there. There the fields $KH_i$ were completed at $Q$ and an intermediate field $G_1$ of $\widehat{KH_2}/\widehat{K(z_2)}$ was introduced such that the extension $G_1/\widehat{K(z_2)}$ is cyclic of degree $q^{n-k}-1$, while the extension $\widehat{KH_2}/G_1$ is a $2$-bounded Galois $p$-extension. Finally the field $G_i=G_1\widehat{KH_i}$ was defined.

Now let us denote by $Q_2$ the restriction of $Q$ to $\widehat{KH_2}$. We obtain from \cite[Figures 9 and 11]{BBGS} that $$e(Q|P_0)=e(Q|Q_2)e(Q_2|P_0)=e(Q|Q_2)q^{n-k-1}(q^k-1).$$
Further denote the restrictions of $Q$ to $G_i$ by $S_i$. Also by \cite[Figures 9 and 11]{BBGS} we have $e(S_i|S_1)=q^{(i-2)(n-k)}$ and $e(Q_2|S_1)=q^{k-1}$. Since
$$e(Q|Q_2)q^{k-1}=e(Q|Q_2)e(Q_2|S_1)=e(Q|S_1)=e(Q|S_i)e(S_i|S_1)=e(Q|S_i)q^{(i-2)(n-k)},$$ and the extensions $\widehat{KH_2}/G_1$ and $G_i/G_1$ are $2$-bounded Galois $p$-extensions, we obtain that $e(Q|Q_2)$ is the product of $q^{(i-2)(n-k)-k+1}$ with a power of the characteristic $p$. Combining the above, we see that $e(Q|P_0)$ is a power of $p$ times $(q^k-1)q^{(i-1)(n-k)-k}$. This proves first part of the proposition.

For the place $P_{\infty}$, we see from Proposition~\ref{prop:ramificationbehavior} that $(q^{n-k}-1)|e_i(P_\infty)$. On the other hand, since any place of $H_i$ lying above $P_{\infty}$ has ramification index $(q^{n-k})^{i-1}$, we have $q^{(n-k)\cdot (i-1)}|e_i(P_\infty)$. Hence $(q^{n-k}-1)\cdot q^{(n-k)(i-1)}$ divides $e_i(P_\infty)$.
\end{proof}

\begin{remark}\label{remark:ram}
Note that by Proposition \ref{prop:ramificationbehavior} the extension $N_i/H_1$ is weakly ramified.
\end{remark}

\begin{proposition}\label{prop:genusHtilde}
We have $$\frac{g(N_i)-1}{[N_i:N_1]} \le \frac12 \cdot \left( \frac{1}{q^k-1}+\frac{1}{q^{n-k}-1} \right).$$
\end{proposition}
\begin{proof}
Denote by $P_0$ (respectively $P_\infty$) the zero (respectively pole) of $z_1$ in $H_1$. We will use the Riemann-Hurwitz formula to estimate the genus of $N_i$. Since only the pole and zero of $z_1$ ramify in the extension $N_i/H_1$, we only need to estimate the different of these places in the extension. Let $\tilde{Q}$ be a place of $N_i$ lying above $P_0$. After completing denote by $S$ the restriction of $\tilde{Q}$ to the intermediate field $L$ from Proposition \ref{prop:ramificationbehavior}. We obtain that
$$e(\tilde{Q}|P_0)=e(\tilde{Q}|S)e(S|P_0)=e(\tilde{Q}|S)(q^k-1)$$
and
$$d(\tilde{Q}|P_0)=e(\tilde{Q}|S)d(S|P_0)+d(\tilde{Q}|S)=e(\tilde{Q}|S)(q^k-2)+2e(\tilde{Q}|S)-2=q^k e(\tilde{Q}|S)-2.$$
Similarly for a place $\tilde Q$ above $P_\infty$ we find
$$e(\tilde{Q}|P_\infty)=e(\tilde{Q}|S)(q^{n-k}-1)$$
and
$$d(\tilde{Q}|P_\infty)=e(\tilde{Q}|S)(q^{n-k}-2)+2e(\tilde{Q}|S)-2=q^{n-k} e(\tilde{Q}|S)-2.$$
We see that
\begin{equation}
\label{boundedzero}
\frac{d(\tilde{Q}|P_0)}{e(\tilde{Q}|P_0)}\le 1+\frac{1}{q^k-1}
\end{equation}
and
\begin{equation}
\label{boundedinfinity}
\frac{d(\tilde{Q}|P_\infty)}{e(\tilde{Q}|P_\infty)}\le 1+\frac{1}{q^{n-k}-1}.
\end{equation}
Using Equations \eqref{boundedzero} and \eqref{boundedinfinity} together with the Riemann-Hurwitz genus formula and the fundamental equality for the extension $N_i/H_1$, the result follows.
\end{proof}

\bigskip

We immediately obtain the following:

\begin{corollary}\label{cor:limit}
The limit of the tower $\mathcal{N}$ satisfies
$$\lambda(\mathcal{N}) \ge 2\left(\frac{1}{q^k-1}+\frac{1}{q^{n-k}-1}\right)^{-1}.$$
\end{corollary}
\begin{proof}
By Proposition \ref{prop:splittingH}, the place $[z_1=-1]$ of $H_1$ splits completely in the tower $\mathcal{H}$ and hence also in the tower $\mathcal{N}$. This together with Proposition \ref{prop:genusHtilde} implies the result.
\end{proof}

\bigskip

At this point we have proved all statements of Theorem~\ref{thm:collecting}, except ii).
\begin{remark}
 Estimates for the limits of the Galois closures $\tilde{\mathcal{E}}$ and $\tilde{\mathcal{F}}$ of the towers $\mathcal E$ and $\mathcal F$ can easily be derived from the above. The lower bound given in Corollary \ref{cor:limit} holds for all of them. More precisely, the tower $\tilde{\mathcal{E}}$ is a subtower of $\mathcal{N}$, since the Galois closure is now taken over $E_1=H_2$. Therefore $\lambda(\tilde{\mathcal{E}}) \ge \lambda(\mathcal{N})$. Lifting the tower $\mathcal{N}$ by adjoining the element $x_1$, gives a Galois tower over $F_1$. By a direct computation, the limit of this lift is easily seen to satisfy the same lower bound as that given for $\lambda(\mathcal{N})$ in Corollary \ref{cor:limit}. Since $\mathcal{\tilde{F}}$ is a subtower of this lifted tower, its limit also satisfies the same lower bound.
\end{remark}

\section{A recursive tower with Galois steps}

In \cite{JNT} and \cite{moscow}, recursive towers over quadratic and cubic finite fields were introduced, where every step is Galois. In this section we obtain an analogous result over any non-prime finite field. More precisely, we construct a recursive subtower $(H'_2\subset H_3' \subset \cdots)$ of the tower $\mathcal{N}$ such that for any $i>1$ the extension $H'_{i+1}/H'_i$ is a Galois extension with elementary abelian $p$-group as Galois group and such that each ramification in $H'_{i+1}/H'_i$ is $2$-bounded.

Starting with the recursive tower $\mathcal H=(H_1 \subset H_2\subset H_3 \subset \cdots)$ as defined in Section~\ref{sec:one} we will introduce an extension field $M/H_1$ such that the composite tower $\mathcal H'=(H_1 \subset H'_2\subset H_3' \subset \cdots)$ with $H'_i=M\cdot H_i$ has Galois steps and its limit satisfies Inequality~\eqref{limitbound}.

Recall that for $i> 0$ we have
$$z_{i}=\frac{\Tr_{n-k}(u_{i})-a}{\Tr_k(u_{i})^{q^{n-k}}-b}=\frac{\Tr_{n-k}(u_{i-1})^{q^k}-a}{\Tr_k(u_{i-1})-b}.$$
Hence $u_{i}$ is a root of the polynomial
\begin{equation}
\label{polui}
\Tr_{n-k}(T)-z_{i}\cdot \Tr_k(T)^{q^{n-k}}-a+z_{i}\cdot b \in \mathbb F_\ell(z_{i})[T].
\end{equation}

The extension $\mathbb F_\ell(u_{i})/\mathbb F_\ell(z_{i})$ is not Galois, but by Proposition \ref{prop:galoisclosurestep1}, the Galois closure of $\mathbb F_\ell(u_{i})/\mathbb F_\ell(z_{i})$ can be obtained by adjoining to $\mathbb F_\ell(u_{i})$ all roots of the polynomial
\begin{eqnarray}
\label{s1}
f_{i}(T) & := & -z_i^{-1}\Tr_{n-k}(T)+\Tr_k(T)^{q^{n-k}}\notag \\
&=&\Tr_n(T)-(1+z_i^{-1})\Tr_{n-k}(T)\notag\\
&=&\Tr_n(T)-\frac{\Tr_n(u_i)-(a+b)}{\Tr_{n-k}(u_i)-a}\Tr_{n-k}(T).
\end{eqnarray}
Similarly, to obtain the Galois closure of the extension $\mathbb F_\ell(u_{i+1})/\mathbb F_\ell(z_{i+1})$, we need to adjoin all roots of
\begin{eqnarray}
\label{s2}
f_{i+1}(T)&=&\Tr_n(T)-(1+z_{i+1}^{-1})\Tr_{n-k}(T)\notag\\
&=&\Tr_n(T)-\frac{\Tr_n(u_i)-(a+b)}{\Tr_{n-k}(u_i)^{q^k}-a}\Tr_{n-k}(T).
\end{eqnarray}
We will show that for each root of $f_{i}(T)$ we get (using $u_i$) a root of the polynomial $f_{i+1}(T)$ and this will give a one-to-one correspondence between roots of $f_i(T)$ and $f_{i+1}(T)$. This implies that by adjoining all roots of $f_{i}(T)$ to a field containing $u_i$, we get all roots of $f_{i+1}(T)$. Hence inductively, we obtain that by lifting the tower $\mathcal H$ by adjoining all roots of $f_1(T)$, we get a tower with Galois steps. First we need a preparatory lemma:

\begin{lemma}\label{lem:prep}
Assume that $s_{i}$ is a root of $f_{i}(T)$, i.e., assume that
$$\Tr_n(s_{i})=\Tr_{n-k}(s_{i})\cdot \frac{\Tr_n(u_{i})-(a+b)}{\Tr_{n-k}(u_{i})-a}.$$
Then we have
\begin{equation}
\label{identity1}
\left(\frac{\Tr_{k}(s_{i})}{\Tr_{k}(u_{i})-b}\right)^{q^{n-k}}=\frac{\Tr_{n-k}(s_{i})}{\Tr_{n-k}(u_{i})-a}
\end{equation}
and
\begin{equation}
\label{identity2}
\Tr_{n-k}(s_{i})^{q^{k}}=\Tr_{n-k}(s_{i})\cdot \frac{\Tr_n(u_{i})-(a+b)}{\Tr_{n-k}(u_{i})-a}-\Tr_{k}(s_{i}).
\end{equation}
\end{lemma}
\begin{proof}
Since $s_i$ is a root of $f_i(T)$, we have
\begin{eqnarray*}
\Tr_{k}(s_{i})^{q^{n-k}}&=&\Tr_{n-k}(s_{i})\cdot \left( \frac{\Tr_n(u_{i})-(a+b)}{\Tr_{n-k}(u_{i})-a}-1\right)\\
&=&\Tr_{n-k}(s_{i})\cdot  \frac{\Tr_{k}(u_{i})^{q^{n-k}}-b}{\Tr_{n-k}(u_{i})-a}.
\end{eqnarray*}
This implies Equation \eqref{identity1}. Equation \eqref{identity2} follows, since
$$\Tr_{k}(s_{i})+\Tr_{n-k}(s_{i})^{q^{k}}=\Tr_n(s_i)=\Tr_{n-k}(s_{i})\cdot \frac{\Tr_n(u_{i})-(a+b)}{\Tr_{n-k}(u_{i})-a}.$$
\end{proof}

\begin{lemma}[Shifting lemma]
\label{shiftlemma}
If $s_{i}$ is a root of $f_{i}(T)$, then
$$s_{i+1}:=\left(\frac{\Tr_{k}(s_{i})}{\Tr_{k}(u_{i})-b}\right)^q-\left(\frac{\Tr_{k}(s_{i})}{\Tr_{k}(u_{i})-b}\right)\in \mathbb F_\ell(u_i,s_i)$$
is a root of $f_{i+1}(T)$.
\end{lemma}
\begin{proof}
\begin{align*}
\Tr_n(s_{i+1})&=\left(\frac{\Tr_{k}(s_{i})}{\Tr_{k}(u_i)-b}\right)^{q^n}-\left(\frac{\Tr_{k}(s_{i})}{\Tr_{k}(u_i)-b}\right)\\
&=\left(\frac{\Tr_{n-k}(s_{i})}{\Tr_{n-k}(u_i)-a}\right)^{q^{k}}-\frac{\Tr_{k}(s_{i})}{\Tr_{k}(u_i)-b}  &\text { by Equation \eqref{identity1}}\\
&=\frac{\Tr_{n-k}(s_{i})\cdot \frac{\Tr_n(u_{i})-(a+b)}{\Tr_{n-k}(u_{i})-a}-\Tr_{k}(s_{i})}{\Tr_{n-k}(u_i)^{q^{k}}-a}-\frac{\Tr_{k}(s_{i})}{\Tr_{k}(u_i)-b} &\text { by Equation \eqref{identity2}}\\
&=\frac{\Tr_n(u_{i})-(a+b)}{\Tr_{n-k}(u_{i})^{q^{k}}-a}\cdot \left[\frac{\Tr_{n-k}(s_{i})}{\Tr_{n-k}(u_{i})-a}-\frac{\Tr_{k}(s_{i})}{\Tr_{k}(u_i)-b}\right]\\
&=\frac{\Tr_n(u_{i})-(a+b)}{\Tr_{n-k}(u_{i})^{q^{k}}-a}\cdot \left[\left(\frac{\Tr_{k}(s_{i})}{\Tr_{k}(u_{i})-b}\right)^{q^{n-k}}-\frac{\Tr_{k}(s_{i})}{\Tr_{k}(u_i)-b}\right]& \text { by Equation \eqref{identity1}}\\
&=\frac{\Tr_n(u_{i})-(a+b)}{\Tr_{n-k}(u_{i})^{q^{k}}-a}\cdot \Tr_{n-k}(s_{i+1}).\\
\end{align*}
Now we see from Equation~\eqref{s2} that
$$f_{i+1}(s_{i+1})=\Tr_n(s_{i+1})-\frac{\Tr_n(u_{i})-(a+b)}{\Tr_{n-k}(u_{i})^{q^{k}}-a}\cdot \Tr_{n-k}(s_{i+1})=0.$$
\end{proof}

We have now established that each root of $f_i(T)$ together with $u_i$ generates a root of $f_{i+1}(T)$. Let $V_i$ (respectively $V_{i+1}$) be the set of roots of $f_i(T)$ (respectively $f_{i+1}(T)$). Since $f_i(T)$ and $f_{i+1}(T)$ are separable and $q$-additive, $V_i$ and $V_{i+1}$ are $(n-1)$-dimensional $\mathbb F_q$-vector spaces. By Lemma~\ref{shiftlemma},
\begin{eqnarray*}
\varphi:V_i & \rightarrow &V_{i+1}\\
s & \mapsto &\left(\frac{\Tr_{k}(s)}{\Tr_{k}(u_i)-b}\right)^q-\left(\frac{\Tr_{k}(s)}{\Tr_{k}(u_i)-b}\right)
\end{eqnarray*}
is a map from $V_i$ to $V_{i+1}$.
Because $\varphi$ is $q$-additive in $s$, it is in fact an $\mathbb F_q$-vector space homomorphism. In fact, it will turn out that $\varphi$ is a bijection.

\begin{lemma}\label{lem:phiinjective}
The map $\varphi: \ V_i\rightarrow V_{i+1}$ defined above is a bijection.
\end{lemma}
\begin{proof}
It is sufficient to show that $\rm{ker}(\varphi)=\{0\}$. Let $s\in V_i$, i.e., $f_i(s)=0$. If
$$\varphi(s)=\left(\frac{\Tr_{k}(s)}{\Tr_{k}(u_i)-b}\right)^q-\left(\frac{\Tr_{k}(s)}{\Tr_{k}(u_i)-b}\right)=0,$$
then $\Tr_{k}(s)/(\Tr_{k}(u_i)-b)\in \mathbb F_q,$
implying that there exists $\alpha \in \F_q$ such that
\begin{equation}
\label{fact1}
\Tr_{k}(s)=\alpha(\Tr_{k}(u_i)-b).
\end{equation}
By Equation \eqref{identity1}, we then have
$$\frac{\Tr_{n-k}(s)}{\Tr_{n-k}(u_{i})-a}=\left(\frac{\Tr_{k}(s)}{\Tr_{k}(u_{i})-b}\right)^{q^{n-k}}=\alpha^{q^{n-k}}=\alpha,$$
so
\begin{equation}
\label{fact2}
\Tr_{n-k}(s)=\alpha(\Tr_{n-k}(u_{i})-a).
\end{equation}
Equations \eqref{fact1} and \eqref{fact2} imply that
$$\Tr_{n-k}(\Tr_k(s))=\alpha \Tr_{n-k}(\Tr_k(u_i)-b)=\alpha\Tr_{n-k}(\Tr_k(u_i))-\alpha \cdot b\cdot (n-k)$$
and
$$\Tr_{k}(\Tr_{n-k}(s))=\alpha \Tr_{k}(\Tr_{n-k}(u_i)-a)=\alpha\Tr_{k}(\Tr_{n-k}(u_i))-\alpha \cdot a\cdot k.$$
Using the above and Lemma \ref{lem1} we obtain
\begin{eqnarray*}
0&=&\Tr_{n-k}(\Tr_k(s))-\Tr_{k}(\Tr_{n-k}(s))\\
&=&\alpha\left(\Tr_{n-k}(\Tr_k(u_i))-\Tr_{k}(\Tr_{n-k}(u_i))+a\cdot k -b\cdot (n-k)\right)\\
&=&\alpha(a\cdot k -b\cdot (n-k))=\alpha.
\end{eqnarray*}
In the last step we used Equation \eqref{eq1}. Equations \eqref{fact1} and \eqref{fact2} now imply that $\Tr_{n-k}(s)=0$ and $\Tr_k(s)=0$. Using Equation \eqref{expleuclid} we conclude that $s=0$.
\end{proof}

By the shifting lemma (Lemma \ref{shiftlemma}) and Lemma \ref{lem:phiinjective} all roots of $f_i(T)$ together with $u_i$ generate all roots of $f_{i+1}(T)$. Similarly all roots of $f_{i+1}(T)$ together with $u_{i+1}$ generate all roots of $f_{i+2}(T)$, etc. So, lifting the tower $\mathcal H$ by the splitting field of $f_1(T)$ gives a tower with Galois steps (see also Proposition \ref{prop:galoisclosurestep1}). More formally, denote by $M$ be the splitting field of $f_1(T)$ over $H_1$ and define $H_i'=M\cdot H_i$ for $i\geq 2$. Then we consider the tower $\mathcal{H}'=(H_1 \subset H_2' \subset H_3' \subset \cdots)$. Note that by Remark \ref{rem:splittingingalois}, the tower $\mathcal{H}'$ is a subtower of $\mathcal{N}$ and moreover $N_2=H_2'$. Note also that all roots of $f_i(T)$ belong to $H_i'$.
\begin{proposition}
\label{galgroups}
\begin{enumerate}
\item All steps in the tower $\mathcal{H}'$ are Galois.
\item The Galois group of the extension $H_2'/H_1$ is an extension by an elementary abelian $p$-group of a subgroup of $\mathrm{GL}_{n-1}(\F_q)$.
\item For each $i>1$, the extension $H_{i+1}'/H_i'$ is an elementary abelian $p$-extension.
\end{enumerate}
\end{proposition}
\begin{proof}
By Proposition \ref{prop:galoisclosurestep1} the field $H_2'=M\cdot H_2$ is a Galois extension of $H_1$. The extension $M/H_1$, being the splitting field of the $q$-additive polynomial $f(T)$ of degree $q^{n-1}$, is Galois with Galois group a subgroup of $\mathrm{GL}_{n-1}(\F_q)$. Since $H_2=H_1(u_1)$, $u_1$ is a root of $f(T)+az_1^{-1}-b$ and $M$ contains all roots of the additive polynomial $f(T)$, the Galois group of $H_2'/M$ is an elementary abelian $p$-group. This proves the second part of the proposition. Similarly, since $H_{i+1}'=H_i'(u_{i})$, $u_{i}$ is a root of $f_{i}(T)+az_{i}^{-1}-b$ and $H_{i}'$ contains all roots of $f_{i}(T)$, the extension $H_{i+1}'/H_i'$ for each $i>1$ is Galois with an elementary abelian $p$-group as Galois group.\end{proof}

\begin{remark}
Note that the tower $(H_2' \subset H_3'\subset \cdots)$ is a recursive tower whose steps are $2$-bounded elementary abelian $p$-extensions (starting at a non-rational function field). Let $E:=M(x_1)$. The composite $E\cdot {\mathcal F}=(E\cdot F_1 \subset E\cdot F_2\subset \cdots )$ is then also a tower whose steps are weakly ramified elementary abelian $p$-extensions. Since both towers are subtowers of $\mathcal{N}$, the bound from Corollary \ref{cor:limit} applies.

\end{remark}

\begin{remark}\label{rem:dualshift}
The very same reasoning applies to a dual tower, by replacing $k$ and $b$ by $n-k$ and $a$, respectively. So a modified version of the shifting lemma and of Proposition~\ref{galgroups} apply in the dual direction.
\end{remark}

The splitting fields over $H_1$ of the polynomials $f_1(T)=f(T)$ and $g(T)$ from Equations~\eqref{hom} and \eqref{homd} are the same. Combining this with Lemma \ref{shiftlemma} and Remark \ref{rem:dualshift}, we see that after adjoining the roots of $f(T)$ to $\F_\ell(z_1)$, all extensions $M(u_{-i},\dots,u_1)/M(u_{-(i-1)},\dots,u_1)$ become Galois. Note that allowing indices $i\leq 0$ in Equation~\eqref{tow2} corresponds to a dual tower.

Since $H_i \subseteq H_i'=M\cdot H_i\subseteq N_i$ for $i>1$, it follows that the Galois closure of $H'_i/H_1$ is given by $N_i$ (the Galois closure of the tower $\mathcal H'$ is the tower $\mathcal{N}$). This observation enables us to describe the Galois group of $N_i/N_1$ and to determine the ramification in the extensions $H'_{i+1}/H'_i$.
The Galois closure of $H'_i/H_1$ is obtained by taking the composite over $H_1$ of $\sigma(H'_i)$ where $\sigma$ runs over all embeddings over $H_1$ of $H'_i$ into a separable closure of $H_1$. Since $H'_2/H_1$ is Galois, we have $\sigma(H_2')=H_2'$ and hence this amounts to taking the composite over $H'_2$ of the $\sigma(H'_i)$. Since all extensions $\sigma(H'_i)/H'_2$ are stepwise Galois $p$-extensions, we see that the extension $N_i/H'_2$ is a Galois $p$-extension.

\bigskip
So we have:
\begin{proposition} \label{rem:groups} The Galois group of $N_i/N_1$ is an extension of a subgroup of $\mathrm{GL}_{n-1}(\F_q)$ by a $p$-group.
\end{proposition}

We can now determine the ramification behavior in the extensions $H'_{i+1}/H'_i$, $i>1$. We have $N_{i+1}\supseteq H'_{i+1}\supseteq H'_i\supseteq H'_2$. Since the extension $N_{i+1}/H'_2$ is a $p$-extension, so is the extension $N_{i+1}/H'_i$. Moreover $N_{i+1}/H'_i$ is weakly ramified, hence $2$-bounded by Remark~\ref{remark:ram}. The $2$-boundedness of $H'_{i+1}/H'_i$ now follows from \cite[Proposition 10]{galclos}. Hence we obtain the following

\begin{proposition} For all $i>1$, the steps $H'_{i+1}/H'_i$ are $2$-bounded Galois $p$-extensions.
\end{proposition}

Collecting all results above, we finish the proof of Theorem \ref{thm:collecting}.

%

\noindent Alp Bassa\\
Sabanc{\i} University, MDBF\\
{\rm 34956} Tuzla, \.Istanbul, Turkey\\
bassa@sabanciuniv.edu\\

\noindent Peter Beelen\\
Technical University of Denmark, Department of Applied Mathematics and Computer Science\\
Matematiktorvet, Building 303B\\
DK-2800, Lyngby, Denmark\\
p.beelen@mat.dtu.dk\\

\noindent
Arnaldo Garcia\\
Instituto Nacional de Matem\'atica Pura e Aplicada, IMPA\\
Estrada Dona Castorina 110\\
22460-320, Rio de Janeiro, RJ, Brazil\\
garcia@impa.br\\

\noindent Henning Stichtenoth\\
Sabanc{\i} University, MDBF\\
{\rm 34956} Tuzla, \.Istanbul, Turkey\\
henning@sabanciuniv.edu\\

\end{document}